\documentclass[12pt]{article}

\usepackage{amsmath}
\usepackage{amsfonts}
\usepackage{amsthm}
\usepackage{graphics}
\newtheorem{theorem}{Theorem}[section]
\newtheorem{lemma}{Lemma}[section]
\numberwithin{equation}{section}

\begin{document}
\title{Asymptotics for the ratio and the zeros of multiple Charlier polynomials
\thanks{Research supported by FWO grant G.0427.09 and K.U.Leuven research grant OT/08/033.
F.N. is supported by a VLIR/UOS scholarship.}}
\author{Fran\c{c}ois Ndayiragije and Walter Van Assche \\ Katholieke Universiteit Leuven}
\date{\today}
\maketitle

\begin{abstract}
We investigate multiple Charlier polynomials and in particular we will use the (nearest neighbor) recurrence
relation to find the asymptotic behavior of the ratio of two multiple Charlier polynomials. This result
is then used to obtain the asymptotic distribution of the zeros, which is uniform on an interval.
We also deal with the case where one of the parameters of the various Poisson distributions depend on the degree
of the polynomial, in which case we obtain another asymptotic distribution of the zeros.
\end{abstract}

\section{Introduction}

Charlier polynomials $\{C_n^{(a)}, n = 0,1,2,\ldots\}$ are orthogonal polynomials for the Poisson distribution, i.e.,
\[   \sum_{k=0}^\infty C_n^{(a)}(k) C_m^{(a)}(k) \frac{a^k}{k!} = 0, \qquad n \neq m, \]
where $a > 0$. These polynomials are orthogonal on the positive integers and as a result their zeros are separated by the integers:
between two consecutive integers there can be at most one zero of $C_n^{(a)}$. Charlier polynomials have various applications, e.g.,  in
queueing theory \cite{KM} and recently \cite{vD}, in the analysis of the lengths of weakly increasing subsequences of random words \cite{Joh}, and in the totally asymmetric simple exclusion process (TASEP) \cite{BFPS}. Their asymptotic behavior has been studied by Maejima and Van Assche \cite{MaVA}, Kuijlaars and
Van Assche \cite{KuiVA}, Rui and Wong \cite{RW}, Goh \cite{Goh}, Dunster \cite{Dun} and most recently by Ou and Wong \cite{Ou} using the Riemann-Hilbert method.

We will investigate multiple Charlier polynomials, which are polynomials of one variable with orthogonality properties with respect to
more than one Poisson distribution. 
Take $r$ Poisson distributions with parameters $a_1, \ldots, a_r >0$ and such that $a_i \neq a_j$ whenever $i \neq j$. Let
 $\vec{n} = (n_1,n_2,\ldots,n_r)$ be a multi-index of size $|\vec{n}| = n_1+n_2+\cdots+n_r$, then the multiple Charlier polynomial
 $C_{\vec{n}}$ is the monic polynomial of degree $|\vec{n}|$ for which (\cite[p.~29--32]{ACV}, \cite[p.~632]{Ismail}, \cite{WVA})
 \[    \sum_{k=0}^\infty C_{\vec{n}}(k) k^\ell \frac{(a_j)^k}{k!} = 0, \qquad \ell=0,1,\ldots,n_j-1, \ j=1\ldots,r. \] 
For $r=1$ we retrieve the Charlier polynomials. The multiple Charlier polynomials can be obtained using the
Rodrigues formula \cite{ACV, Ismail, WVA}  
 \begin{equation}  \label{Rodrigues}
     C_{\vec{n}}(x) = (-1)^{|\vec{n}|} \left( \prod_{j=1}^r a_j^{n_j} \right) 
     \Gamma(x+1) \left( \prod_{j=1}^r  a_j^{-x} \nabla^{n_j} a_j^x \right) \frac{1}{\Gamma(x+1)}  
  \end{equation}
 where $\nabla$ is the backward difference operator, given by $\nabla f(x) = f(x)-f(x-1)$. An explicit formula for the multiple Charlier polynomials is
\begin{multline}  \label{explicit}
      C_{\vec{n}}(x) = \sum_{k_1=0}^{n_1} \cdots \sum_{k_r=0}^{n_r} (-n_1)_{k_1} \cdots (-n_r)_{k_r} (-x)_{k_1+k_2+\cdots+k_r}  \\
      \frac{(-a_1)^{n_1-k_1} (-a_2)^{n_2-k_2} \cdots (-a_r)^{n_r-k_r}}{k_1! k_2! \cdots k_r!}.   
\end{multline}

Multiple Charlier polynomials satisfy a number of (higher order) difference equations (Lee \cite{Lee} and Van Assche \cite{WVA}). 
They appear in remainder Pad\'e approximation for the exponential function \cite{PR}, as common eigenstates of a set of $r$ non-Hermitian oscillator
Hamiltonians \cite{Vinet}, and we believe that they are related to the orthogonal functions
appearing in two speed TASEP (totally asymmetric simple exclusion process) \cite{BFS}. 

In this paper we first obtain in Section 2 some properties of the multiple Charlier polynomials, such as the generating function and
the nearest neighbor recurrence relations. The zeros of multiple Charlier polynomials are real, positive and separated by the positive integers,
as is the case for the usual Charlier polynomials: between two positive integers, there can be at most one zero of a multiple Charlier polynomial
(see, e.g., \cite[Theorem 3.4]{PVA}). The largest zero of $C_{\vec{n}}$ is therefore $\geq |\vec{n}|-1$. In order to prevent the zeros to
go to infinity, we will use a scaling and consider the scaled polynomials $P_{\vec{n},N}(x) = C_{\vec{n}}(Nx)/N^{|\vec{n}|}$.
One of the main results in this paper is in Section 3 where we obtain the asymptotic behavior of the ratio of two scaled neighboring 
multiple Charlier polynomials. We use that result in Section 4 to obtain the asymptotic zero distribution of the scaled multiple Charlier polynomials. Another important result is in Secton 5 where we give the asymptotic behavior
(ratio asymptotics and zero distribution) when one of the parameters depends on the scaling $N$. This gives a different asymptotic zero distribution which is somewhat more interesting.

\section{Some properties of multiple Charlier polynomials}

\subsection{Generating function}
Charlier polynomials have the generating function \cite[Ch. VI, Eq.~(1.1)]{Chihara}
\begin{equation}  \label{GenChar}    
\sum_{n=0}^\infty C_n^{(a)}(x) \frac{t^n}{n!} = (1+t)^x e^{-at}, \qquad |t| < 1.  
\end{equation}
For multiple Charlier polynomials one has a multivariate generating function (with $r$ variables).

\begin{theorem}  \label{thm:gen}
Multiple Charlier polynomials have the following (multivariate) generating function
\begin{multline}  \label{generating}
   \sum_{n_1=0}^\infty \sum_{n_2=0}^\infty \cdots \sum_{n_r=0}^\infty  C_{\vec{n}}(x) \frac{t_1^{n_1} t_2^{n_2} \cdots 
    t_r^{n_r}}{n_1! n_2! \cdots n_r!}  \\
  = (1+t_1+t_2+\cdots+t_r)^x \exp (-a_1t_1-a_2t_2 - \cdots -a_r t_r). 
\end{multline} 
\end{theorem}  
\begin{proof}
We can use induction on $r$. For $r=1$ we have the familiar generating function for Charlier polynomials \eqref{GenChar}. 

Suppose the result is true for $r-1$, then observe that (\ref{explicit}) implies
\[  C_{\vec{n}}(x) = \sum_{k_r=0}^{n_r} C_{\vec{n}-n_r\vec{e}_r}(x-k_r) (-x)_{k_r} (-n_r)_{k_r} \frac{(-a_r)^{n_r-k_r}}{k_r!}.  \] 
Hence the multivariate generating function is
\begin{multline*}   (1+t_1+\cdots+t_{r-1})^{x} \exp(-a_1t_1-\cdots-a_{r-1}t_{r-1}) \\
\sum_{n_r=0}^\infty \sum_{k_r=0}^{n_r} 
  (1+t_1+\cdots+t_{r-1})^{-k_r} (-x)_{k_r} (-n_r)_{k_r} \frac{t_r^{n_r}}{n_r!} \frac{(-a_r)^{n_r-k_r}}{k_r!}. 
\end{multline*} 
Changing the order of summation gives
\begin{multline*}
(1+t_1+\cdots+t_{r-1})^{x} \exp(-a_1t_1-\cdots-a_{r-1}t_{r-1}) \\
 \sum_{k_r=0}^\infty (1+t_1+\cdots+t_{r-1})^{-k_r} \binom{x}{k_r} 
   \sum_{n_r=k_r}^{\infty}  \frac{t_r^{n_r}}{(n_r-k_r)!} (-a_r)^{n_r-k_r} 
\end{multline*}
and by putting $\ell=n_r-k_r$
\begin{multline*}  (1+t_1+\cdots+t_{r-1})^{x} \exp(-a_1t_1-\cdots-a_{r-1}t_{r-1}) \\
\sum_{k_r=0}^\infty (1+t_1+\cdots+t_{r-1})^{-k_r} t_r^{k_r} \binom{x}{k_r}
   \sum_{\ell=0}^{\infty}  \frac{t_r^{\ell}}{\ell!} (-a_r)^{\ell}.
\end{multline*}
Now use 
\[ \sum_{\ell=0}^{\infty}  \frac{t_r^{\ell}}{\ell!} (-a_r)^{\ell} = \exp(-a_r t_r)  \]
and
\[  \sum_{k_r=0}^\infty (1+t_1+\cdots+t_{r-1})^{-k_r} t_r^{k_r} \binom{x}{k_r}  =  \left(1+ \frac{t_r}{1+t_1+t_2+\cdots+t_{r-1}} \right)^x \]
to obtain the desired result.
\end{proof}

The region of convergence of this generating function is a log-convex set in $\mathbb{C}^r$, which is the case of all power series in several variables, and the series certainly converges whenever $|t_j| < 1/r$ for every $j \in \{ 1,2,\ldots,r\}$, or when $|t_j| < c_j$ for $1 \leq j \leq r$, where
$0 < c_j < 1$ and $\sum_{j=1}^r c_j = 1$.
As a corollary, one can obtain an integral representation of the multiple Charlier polynomial, by integrating $r$ times over a closed curve around $0$:
\[   \frac{C_{\vec{n}}(x)}{n_1! \cdots n_r!} 
      = \frac{1}{(2\pi i)^r} \oint \cdots \oint \frac{(1+z_1+\cdots+z_r)^x \exp(-a_1 z_1 - \cdots - a_r z_r)}
                             {z_1^{n_1+1} \cdots z_r^{n_r+1}} \ dz_1\cdots dz_r.  \]

\subsection{Recurrence relations}
For multiple orthogonal polynomials there is always a nearest neighbor recurrence relation
 of the form
\begin{equation}  \label{eq:recur}
   xP_{\vec{n}}(x) = P_{\vec{n}+\vec{e_k}}(x) + b_{\vec{n},k} P_{\vec{n}}(x)
    + \sum_{j=1}^r a_{\vec{n},j} P_{\vec{n}-\vec{e}_j}(x)  
\end{equation}
where $k=1,\ldots,r$ \cite[Thm.~ 23.1.11]{Ismail}, \cite{WVA2}, and $\vec{e}_k = (0,\ldots,0,1,0,\ldots,0)$ is the $k$th unit vector in $\mathbb{N}^r$.
The recurrence relation for multiple Charlier polynomials was given in \cite[p.~632]{Ismail} without proof. Here we will work out the details of the proof.
\begin{theorem}  \label{thm:rec}  
The nearest neighbor recurrence relation for multiple Charlier polynomials is  
\begin{equation}  \label{Char:recur}
   x C_{\vec{n}}(x) = C_{\vec{n}+\vec{e}_k}(x) + (a_k + |\vec{n}|) C_{\vec{n}}(x) + \sum_{j=1}^r n_j a_j C_{\vec{n}-\vec{e}_j}(x).  
\end{equation}
\end{theorem}

\begin{proof}
 From (\ref{explicit}) and $(-x)_n = (-1)^n x^n + (-1)^{n-1} \binom{n}{2} x^{n-1} + \cdots$
we find that
\[ C_{\vec{n}}(x) = x^{|\vec{n}|} + \delta_{\vec{n}} x^{|\vec{n}|-1} + \cdots, \]
where $\delta_{\vec{n}}$ can be found by taking $(k_1,k_2,\ldots,k_r)=(n_1,n_2,\ldots,n_r)$, which gives the contribution $-\binom{|\vec{n}|}{2}$ to $\delta_{\vec{n}}$, and for each $j$ with $1 \leq j \leq r$ we get
for $(k_1,k_2,\ldots,k_r) = (n_1,n_2,\ldots, n_j-1, \ldots, n_r)$ the contribution $-a_jn_j$, so that
\[   \delta_{\vec{n}} = - \binom{|\vec{n}|}{2} - \sum_{j=1}^r a_j n_j.  \]
If we compare the coefficient of $x^{|\vec{n}|}$ in (\ref{eq:recur}), then $b_{\vec{n},k}= \delta_{\vec{n}}
-\delta_{\vec{n}+\vec{e}_k}$, which for the multiple Charlier polynomials gives $b_{\vec{n},k} = |\vec{n}| + a_k$.
For the recurrence coefficients $a_{\vec{n},j}$ we can use \cite[Eq. (23.1.23)]{Ismail} 
\[   a_{\vec{n},j} = \frac{ \sum_{k=0}^\infty k^{n_j} C_{\vec{n}}(k) a_j^k/k! }
                         { \sum_{k=0}^\infty k^{n_j-1} C_{\vec{n}-\vec{e}_j}(k) a_k^k/k!}.  \]
The sums can be computed using the Rodrigues formula (\ref{Rodrigues}): the difference operators
$a_i^{-x} \nabla^{n_i} a_i^x$ $(i=1,2,\ldots,r)$ are commuting, so we can first apply  $a_j^{-x} \nabla^{n_j} a_j^x$ to find
\[   \sum_{k=0}^\infty k^{n_j} C_{\vec{n}}(k) \frac{a_j^k}{k!} = 
     (-1)^{|\vec{n}|} \prod_{i=1}^r a_i^{n_i}  \sum_{k=0}^\infty k^{n_j} \left( \nabla^{n_j} a_j^k \right)
     \left( \prod_{i=1,i\neq j}^r a_i^{-k} \nabla^{n_i} a_i^k \right) \frac{1}{k!}.  \]
Now use summation by parts $n_k$ times to find  
\begin{eqnarray*}
 \sum_{k=0}^\infty k^{n_j} C_{\vec{n}}(k) \frac{a_j^k}{k!} & = &
    (-1)^{|\vec{n}|} \prod_{i=1}^r a_i^{n_i}  
   (-1)^{n_j} \sum_{k=0}^\infty \left( \Delta^{n_j} k^{n_j} \right) a_j^k  
\left( \prod_{i=1,i\neq j}^r a_i^{-k} \nabla^{n_i} a_i^k \right) \frac{1}{k!} \\
  & = & (-1)^{|\vec{n}|} \prod_{i=1}^r a_i^{n_i}  
   (-1)^{n_j} n_j! \sum_{k=0}^\infty  a_j^k  
\left( \prod_{i=1,i\neq j}^r a_i^{-k} \nabla^{n_i} a_i^k \right) \frac{1}{k!}.
\end{eqnarray*}
If we change $n_j$ to $n_j-1$ then this gives 
\begin{multline*}
  \sum_{k=0}^\infty k^{n_j-1} C_{\vec{n}-\vec{e}_j}(k) \frac{a_j^k}{k!} \\
     =  (-1)^{|\vec{n}|-1} \left( \prod_{i=1}^r a_i^{n_i} \right) a_j^{-1}  
   (-1)^{n_j-1} (n_j-1)! \sum_{k=0}^\infty  a_j^k  
\left( \prod_{i=1,i\neq j}^r a_i^{-k} \nabla^{n_i} a_i^k \right) \frac{1}{k!}.  
\end{multline*}
Dividing both expressions then gives
\[   a_{\vec{n},j} = n_j a_j   .   \]
\end{proof}

The recurrence coefficients are quite simple in this case, and in particular $a_{\vec{n},j} = n_j a_j > 0$ whenever $n_j \in \mathbb{N}$.
This implies that the zeros of $C_{\vec{n}}$ and its nearest neighbors $C_{\vec{n}+\vec{e}_k}$ interlace for every $k \in \{1,2,\ldots,r\}$,
see \cite{HVA}. This will be useful in the next section.

\section{Ratio asymptotics}

There are various levels of asymptotic behavior to consider. In this paper we limit the analysis to ratio asymptotic behavior, i.e., the
asymptotic behavior of the ratio of two neighboring polynomials. In order to prevent the zeros from going to infinity, we use a scaling
and we will investigate the ratio $C_{\vec{n}+\vec{e}_k}(Nx)/C_{\vec{n}}(Nx)$ for $x \in \mathbb{C} \setminus [0,\infty)$, where
$N$ is of the order $|\vec{n}|$, i.e., $\lim_{N \to \infty} |\vec{n}|/N = t > 0$.

\begin{theorem}   \label{Thm1}
Suppose $n_j=\lfloor q_jn\rfloor$, with $0<q_j <1$ and $\sum_{j=1}^r q_j=1$, so that $|\vec{n}|/n \to 1$ as $n \to \infty$.
Let $a_i >0$ for $1 \leq i \leq r$ and $a_i \neq a_j$ 
whenever $i \neq j$. Then for $t >0$ and for every $k \in \{1,2,\ldots,r\}$ one has
\begin{equation}  \label{ratio}
    \lim_{n \to \infty, n/N \to t} \frac{C_{\vec{n}+\vec{e}_k}(Nx)}{N C_{\vec{n}}(Nx)} = x-t  
\end{equation}
uniformly for $x \in K$, where $K$ is a compact set in $\mathbb{C} \setminus [0,\infty)$. 
\end{theorem}

\begin{proof}
We will use the notation $P_{\vec{n},N}(x) = C_{\vec{n}}(Nx)/N^{|\vec{n}|}$ for the monic and rescaled multiple Charlier polynomials.
The zeros of $C_{\vec{n}-\vec{e}_j}$ and $C_{\vec{n}}$ are real, positive and interlace (since $a_{\vec{n},j} = a_j n_j > 0$ whenever $n_j > 0$,
see \cite{HVA}), hence we have the partial fractions decomposition
\[     \frac{P_{\vec{n}-\vec{e}_j,N}(x)}{P_{\vec{n},N}(x)} = \sum_{i=1}^{|\vec{n}|} \frac{A_{\vec{n},i}}{x- x_{\vec{n},i}/N}, \]
where $\{x_{\vec{n},i} : 1 \leq i \leq |\vec{n}|\}$ are the zeros of $C_{\vec{n}}$ and $A_{\vec{n},i} > 0$ for every $i \leq |\vec{n}|$.
Let $K$ be a compact set in $\mathbb{C} \setminus [0, \infty)$, then for $x \in K$ we have that
\[ \left|  \frac{P_{\vec{n}-\vec{e}_j,N}(x)}{P_{\vec{n},N}(x)} \right|
   \leq \sum_{i=1}^{|\vec{n}|} \frac{A_{\vec{n},i}}{|x- x_{\vec{n},i}/N|} \leq  \frac{1}{\delta} \sum_{i=1}^{|\vec{n}|} A_{\vec{n},i}, \]
where 
\[    \delta = \inf \{ |z-y| : z \in K, y \in [0,\infty) \} > 0  \]
is the minimal distance between $K$ and $[0,\infty)$. Since $P_{\vec{n},N}$ and $P_{\vec{n}-\vec{e}_j,N}$ are monic polynomials, one has
$\sum_{i=1}^{|\vec{n}|} A_{\vec{n},i} = 1$, so that we have the bound
\begin{equation}   \label{ratiobound}
     \left|  \frac{P_{\vec{n}-\vec{e}_j,N}(x)}{P_{\vec{n},N}(x)} \right|  \leq \frac{1}{\delta}, 
\end{equation}uniformly for $x \in K$.
Take the recurrence relation \eqref{Char:recur} with $x$ replaced by $Nx$, and divide by $C_{\vec{n}}(Nx)$, which is allowed since
$x \in K$ cannot be a zero, then we find
\[     x = \frac{P_{\vec{n}+\vec{e}_k,N}(x)}{P_{\vec{n},N}(x)} + \frac{a_k + |\vec{n}|}{N} 
       + \sum_{j=1}^r   \frac{n_ja_j}{N^2} \frac{P_{\vec{n}-\vec{e}_j,N}(x)}{P_{\vec{n},N}(x)} . \]
If we use the bound \eqref{ratiobound}, then this gives
\[  \left|  \frac{P_{\vec{n}+\vec{e}_k,N}(x)}{P_{\vec{n},N}(x)} - x + \frac{a_k + |\vec{n}|}{N} \right|
    \leq   \frac{1}{\delta} \sum_{j=1}^r  \frac{n_j a_j}{N^2}. \]
Clearly, when $n,N \to \infty$ in such a way that $n/N \to t$, we have
\[    \lim_{n \to \infty, n/N \to t} \frac{a_k + |\vec{n}|}{N} = \lim_{n \to \infty, n/N \to t} \frac{|\vec{n}|}{n} \frac{n}{N} = t, \]
and
\[    \lim_{n \to \infty, n/N \to t} \frac{a_j n_j}{N^2} = a_j \frac{n_j}{n} \frac{n}{N^2} = 0, \]
so that
\[   \lim_{n \to \infty, n/N \to t}  \frac{P_{\vec{n}+\vec{e}_k,N}(x)}{P_{\vec{n},N}(x)} = x- t, \]
uniformly for $x \in K$, which proves the theorem.
\end{proof}

Observe that the same result will hold for any family of multiple orthogonal polynomials for which $a_{\vec{n},j} > 0$ whenever $n_j >0$
and
\[    \lim_{n \to \infty} \frac{b_{\vec{n},k}}{n} = 1, \qquad \lim_{n \to \infty} \frac{a_{\vec{n},j}}{n^2} = 0, \]
where $n_j = \lfloor q_j n \rfloor$, with $0 < q_j < 1$ and $\sum_{j=1}^r q_j = 1$. The fact that $a_{\vec{n},j}/n^2 \to 0$ simplifies
the asymptotic analysis a lot and the limit function is an easy polynomial function of degree 1. In general, the asymptotic analysis
for ratios of multiple orthogonal polynomials would involve a limit function which is the solution of an algebraic equation of degree $r+1$.

\section{Asymptotic distribution of the zeros}

Next, we will obtain the asymptotic distribution of the (scaled) zeros of the multiple Charlier polynomials. For this, we introduce
the zero measure
\[   \nu_{n,N} = \frac{1}{|\vec{n}|} \sum_{i=1}^{|\vec{n}|} \delta_{x_{\vec{n},i}/N} \]
and we want to show that these (probability) measures converge weakly to a (probability) measure $\nu_t$ as $n, N \to \infty$ and $n/N \to t > 0$,
which then describes the asymptotic distribution of the zeros. Again we will take multi-indices $\vec{n}$ such that $n_j = \lfloor nq_j \rfloor$,
where $0 < q_j < 1$ and $\sum_{j=1}^r q_j = 1$, so that $|\vec{n}|/n \to 1$ as $n$ tends to infinity. In order to prove this weak convergence,
we will investigate their Stieltjes transform
\[     \int \frac{d \nu_{n,N}(y)}{x-y} = \frac{1}{|\vec{n}|} \frac{P_{\vec{n},N}'(x)}{P_{\vec{n},N}(x)}, \qquad x \in \mathbb{C} \setminus [0,\infty), \]
where $P_{\vec{n},N}(x) = C_{\vec{n}}(Nx)/N^{|\vec{n}|}$, 
and show that they converge to a function, which we can identify as the Stieltjes transform of a measure $\nu_t$. The Grommer-Hamburger
theorem \cite{GH} then tells us that the measures $\nu_{n,N}$ converge weakly to $\nu_t$ as $n,N \to \infty$ and $n/N \to t$.

\begin{theorem} \label{Thm2}
Suppose $n_j=\lfloor q_jn\rfloor$, with $0<q_j <1$ and $\sum_{j=1}^r q_j=1$ and that $a_i >0$ for $1 \leq i \leq r$ and $a_i \neq a_j$ 
whenever $i \neq j$. Let $x_{\vec{n},1} < x_{\vec{n},2} < \cdots < x_{\vec{n},|\vec{n}|}$ be the zeros of $C_{\vec{n}}$. Then
\begin{equation} \label{zerodist}
   \lim_{n,N \to \infty, n/N \to t} \frac{1}{|\vec{n}|} \sum_{j=1}^{|\vec{n}|} f(x_{\vec{n},j}/N) = \frac{1}{t} \int_0^t f(x)\, dx  
\end{equation}
for every bounded continuous function on $[0,\infty)$. This means that the zeros of $C_{\vec{n}}(Nx)$ are asymptotically uniform
on the interval $[0,t]$ when $n,N \to \infty$ and $n/N \to t > 0$.
\end{theorem}

\begin{proof}
We will prove that
\begin{equation}  \label{P'P}
    \lim_{n,N \to \infty, n/N \to t} \frac{1}{|\vec{n}|} \frac{P_{\vec{n},N}'(x)}{P_{\vec{n},N}(x)}
    = \frac{1}{t} \int_0^t \frac{1}{x-y} \, dy  
\end{equation}
uniformly for $x \in K$, where $K$ is a compact set in $\mathbb{C} \setminus [0,\infty)$, which by the Grommer-Hamburger
theorem (see, e.g., \cite{GH}) is equivalent with the weak convergence to the uniform measure on $[0,t]$.
We will prove this by induction on $r$. For $r=1$ we deal with the zeros of Charlier polynomials and the multi-index $\vec{n}$ is an integer
which we denote by $n$. Observe that
\[     \frac{1}{n} \frac{P_{n,N}'(x)}{P_{n,N}(x)} 
  = \frac{1}{n} \sum_{k=0}^{n-1} \left( \frac{P_{k+1,N}'(x)}{P_{k+1,N}(x)} - \frac{P_{k,N}'(x)}{P_{k,N}(x)} \right) \]
and straightforward calculus gives
\[      \frac{P_{k+1,N}'(x)}{P_{k+1,N}(x)} - \frac{P_{k,N}'(x)}{P_{k,N}(x)} 
     = \left( \frac{P_{k+1,N}(x)}{P_{k,N}(x)} \right)' \Big/  \left(\frac{P_{k+1,N}(x)}{P_{k,N}(x)} \right). \]
Hence we may write
\[    \frac{1}{n} \frac{P_{n,N}'(x)}{P_{n,N}(x)} = \frac{1}{n} \sum_{k=0}^{n-1}
     \left( \frac{P_{k+1,N}(x)}{P_{k,N}(x)} \right)' \Big/  \left(\frac{P_{k+1,N}(x)}{P_{k,N}(x)} \right) . \]
We can rewrite the sum as an integral by putting $k = \lfloor ns \rfloor$, so that
\[   \frac{1}{n} \frac{P_{n,N}'(x)}{P_{n,N}(x)} =    \int_0^1 
 \left( \frac{P_{\lfloor ns \rfloor +1,N}(x)}{P_{\lfloor ns \rfloor,N}(x)} \right)' 
    \Big/  \left(\frac{P_{\lfloor ns \rfloor+1,N}(x)}{P_{\lfloor ns \rfloor,N}(x)} \right) \ ds . \]
Now we let $n,N \to \infty$ in such a way that $n/N \to t$, and we use Theorem \ref{Thm1} (with $r=1$) to find that
uniformly for $x \in K$ ($K$ a compact set in $\mathbb{C} \setminus [0,\infty)$)
\[   \lim_{n,N \to \infty, n/N \to t} \frac{1}{n} \frac{P_{n,N}'(x)}{P_{n,N}(x)}
     = \int_0^1  \frac{(x-st)'}{x-st} \ ds, \]
where the $'$ in the integral is a derivative with respect to the variable $x$.
The integral on the right is (use $st=y$)
\[   \int_0^1  \frac{1}{x-st} \, ds  = \frac{1}{t} \int_0^t \frac{1}{x-y} \, dy , \]
which proves \eqref{P'P} for $r=1$.

Now suppose that \eqref{P'P} is true for $r-1$. Observe that
\begin{equation}  \label{r}
     \frac{P_{\vec{n},N}'(x)}{P_{\vec{n},N}(x)} =
      \frac{P_{\vec{n}-n_r\vec{e}_r,N}'(x)}{P_{\vec{n}-n_r\vec{e}_r,N}(x)}
       + \sum_{k=0}^{n_r-1}   \left( \frac{P_{\vec{n}-k\vec{e}_r,N}(x)}{P_{\vec{n}-(k+1)\vec{e}_r,N}(x)} \right)' 
          \Big/  \left(\frac{P_{\vec{n}-k\vec{e}_r,N}(x)}{P_{\vec{n}-(k+1)\vec{e}_r,N}(x)} \right). 
\end{equation}
The multiple orthogonal polynomial $P_{\vec{n}-n_r\vec{e}_r}$ is in fact a multiple orthogonal polynomial with only $r-1$
measures $(\mu_1,\ldots,\mu_{r-1})$, hence we can use the induction hypothesis to find
\begin{equation}  \label{r-1}
   \lim_{n,N \to \infty} \frac{1}{|\vec{n}|-n_r} \frac{P_{\vec{n}-n_r\vec{e}_r,N}'(x)}{P_{\vec{n}-n_r\vec{e}_r,N}(x)}
      =   \frac{1}{(1-q_r)t} \int_0^{(1-q_r)t} \frac{1}{x-y}\, dy . 
\end{equation}     
Note that $(|\vec{n}|-n_r)/n \to 1 - q_r$, which explains the appearance of $1-q_r$ in the last formula. We can write the sum as an integral
by taking $k = \lfloor n_r s \rfloor$:
\begin{multline*} 
\frac{1}{n_r} \sum_{k=0}^{n_r-1}   \left( \frac{P_{\vec{n}-k\vec{e}_r,N}(x)}{P_{\vec{n}-(k+1)\vec{e}_r,N}(x)} \right)' 
          \Big/  \left(\frac{P_{\vec{n}-k)\vec{e}_r,N}(x)}{P_{\vec{n}-(k+1)\vec{e}_r,N}(x)} \right)  \\
    = \int_0^1  \left( \frac{P_{\vec{n}-\lfloor n_r s \rfloor\vec{e}_r,N}(x)}{P_{\vec{n}-(\lfloor n_r s \rfloor+1)\vec{e}_r,N}(x)} \right)' 
     \Big/  \left(\frac{P_{\vec{n}-\lfloor n_r s \rfloor\vec{e}_r,N}(x)}{P_{\vec{n}-(\lfloor n_r s \rfloor+1)\vec{e}_r,N}(x)} \right)\ ds.  
\end{multline*}
Now use Theorem \ref{Thm1} to find
\begin{multline}  \label{sum}
 \lim_{n,N \to \infty, n/N \to t}  
  \frac{1}{n_r} \sum_{k=0}^{n_r-1}   \left( \frac{P_{\vec{n}-k\vec{e}_r,N}(x)}{P_{\vec{n}-(k+1)\vec{e}_r,N}(x)} \right)' 
          \Big/  \left(\frac{P_{\vec{n}-k\vec{e}_r,N}(x)}{P_{\vec{n}-(k+1)\vec{e}_r,N}(x)} \right)  \\
  = \int_0^1 \frac{(x-(1-q_rs)t)'}{x-(1-q_rs)t} \, ds = \frac{1}{q_rt} \int_{(1-q_r)t}^{t} \frac{1}{x-y} \, dy 
\end{multline}
where the last equality follows after using the substitution $y=(1-q_rs)t$. Note that $(|\vec{n}|-\lfloor n_r s\rfloor)/n \to 1-q_rs$,
which explains the factor $1-q_rs$ in the asymptotic formula. Now combine \eqref{r-1} and \eqref{sum} in \eqref{r} to find
\begin{eqnarray*}
  \lim_{n,N \to \infty, n/N \to t} \frac{1}{|\vec{n}|} \frac{P_{\vec{n},N}'(x)}{P_{\vec{n},N}(x)}
    &=&  \frac{1-q_r}{(1-q_r)t} \int_0^{(1-q_r)t} \frac{1}{x-y} \, dy + \frac{q_r}{q_r t} \int_{(1-q_r)t}^t \frac{1}{x-y}\, dy \\
    &=& \frac{1}{t} \int_0^t \frac{1}{x-y} \, dy, 
\end{eqnarray*}
which proves \eqref{P'P}.
\end{proof}

\section{Parameters depending on the degree}
We get more interesting asymptotics when some of the parameters depend on $N$ and grow together with the degree $|\vec{n}|$. The case where only
one parameter depends on $N$ can be worked out in detail.

\begin{theorem} \label{thm3}
Suppose $n_j = \lfloor q_jn\rfloor$, with $0 < q_j < 1$ and $\sum_{j=1}^r q_j = 1$, so that $|\vec{n}|/n \to 1$ as $n \to <\infty$.
Consider Poisson distributions with parameters $(a_1,a_2,\ldots,a_{r-1},Na_r)$, i.e., the last parameter grows linearly with $N$.
Then for $t>0$ one has
\begin{equation}  \label{gr}
  \lim_{n,N\to \infty,\ n/N \to t} \frac{C_{\vec{n}+\vec{e}_r}(Nx)}{N C_{\vec{n}}(Nx)} =
    \frac{x-a_r-t + \sqrt{(x-a_r-t)^2-4a_rq_r t}}{2} := g_r(x) 
\end{equation}
and for $1 \leq k < r$
\begin{equation} \label{gnotr}
  \lim_{n,N\to \infty,\ n/N \to t} \frac{C_{\vec{n}+\vec{e}_k}(Nx)}{N C_{\vec{n}}(Nx)} =
    x-t- \frac{a_rq_rt}{g_r(x)} 
\end{equation}
uniformly on compact sets of $\mathbb{C} \setminus [0,\infty)$.
\end{theorem}

\begin{proof}
We still use the notation $P_{\vec{n},N}(x) = C_{\vec{n}}(Nx)/N^{|\vec{n}|}$, but now keep in mind that $C_{\vec{n}}$ depends
on the $r$ parameters $(a_1,\ldots,a_{r-1},Na_r)$ so that the parameter $N$ appears not only in the scaling of the variable $(Nx)$ 
but also in the last parameter $(Na_r)$. The recurrence relation \eqref{Char:recur}, after dividing by $C_{\vec{n}}(Nx)$ gives for
$x \in K$, where $K$ is a compact set in $\mathbb{C} \setminus [0,\infty)$,
\[   x = \frac{P_{\vec{n}+\vec{e}_k,N}(x)}{P_{\vec{n},N}(x)} + \frac{a_k+|\vec{n}|}{N} 
 + \sum_{j=1}^{r-1} \frac{n_ja_j}{N^2} \frac{P_{\vec{n}-\vec{e}_j,N}(x)}{P_{\vec{n},N}(x)} 
 + \frac{n_ra_r}{N} \frac{P_{\vec{n}-\vec{e}_r,N}(x)}{P_{\vec{n},N}(x)}  \]
when $1 \leq k \leq r-1$, and for $k=r$ we have
\[   x = \frac{P_{\vec{n}+\vec{e}_r,N}(x)}{P_{\vec{n},N}(x)} + \frac{Na_r+|\vec{n}|}{N} 
 + \sum_{j=1}^{r-1} \frac{n_ja_j}{N^2} \frac{P_{\vec{n}-\vec{e}_j,N}(x)}{P_{\vec{n},N}(x)} 
 + \frac{n_ra_r}{N} \frac{P_{\vec{n}-\vec{e}_r,N}(x)}{P_{\vec{n},N}(x)}.  \]
If we use \eqref{ratiobound}, then for $1 \leq k \leq r-1$
\[  \left| \frac{P_{\vec{n}+\vec{e}_k,N}(x)}{P_{\vec{n},N}(x)} -x + \frac{a_k+|\vec{n}|}{N} + 
   \frac{n_ra_r}{N} \frac{P_{\vec{n}-\vec{e}_r,N}(x)}{P_{\vec{n},N}(x)} \right| \leq \frac{1}{\delta N^2} \sum_{j=1}^{r-1} n_ja_j  \]
and for $k=r$
\[ \left| \frac{P_{\vec{n}+\vec{e}_r,N}(x)}{P_{\vec{n},N}(x)} -x + \frac{Na_r+|\vec{n}|}{N} + 
   \frac{n_ra_r}{N} \frac{P_{\vec{n}-\vec{e}_r,N}(x)}{P_{\vec{n},N}(x)} \right| \leq \frac{1}{\delta N^2} \sum_{j=1}^{r-1} n_ja_j  \]
so that
\[   \lim_{n \to \infty, n/N \to t}  \left| \frac{P_{\vec{n}+\vec{e}_k,N}(x)}{P_{\vec{n},N}(x)} -x + t + 
   a_rq_rt \frac{P_{\vec{n}-\vec{e}_r,N}(x)}{P_{\vec{n},N}(x)} \right| = 0, \qquad 1 \leq k \leq r-1, \]
and
\[   \lim_{n \to \infty, n/N \to t}  \left| \frac{P_{\vec{n}+\vec{e}_r,N}(x)}{P_{\vec{n},N}(x)} -x + a_r + t + 
   a_rq_rt \frac{P_{\vec{n}-\vec{e}_r,N}(x)}{P_{\vec{n},N}(x)} \right| = 0, \]
uniformly on $K$. The bound \eqref{ratiobound} implies that $\{P_{\vec{n}-\vec{e}_j,N}(x)/P_{\vec{n},N}(x): n,N \in \mathbb{N}\}$ is a normal family on every compact
subset of $\mathbb{C} \setminus [0,\infty)$, hence there is a subsequence which converges uniformly on $K$:
\[  \lim_{n_i \to \infty, n_i/N_i \to t} \frac{P_{\vec{n}_i-\vec{e}_j,N_i}(x)}{P_{\vec{n}_i,N_i}(x)} = h_j(x), \]
and, by taking further subsequences, this convergence holds for every $j$ for which $1 \leq j \leq r$. 
With our previous estimates, this gives
\begin{equation}  \label{ratio1}
  \lim_{n_i \to \infty, n_i/N_i \to t} \frac{P_{\vec{n}_i+\vec{e}_j,N_i}(x)}{P_{\vec{n}_i,N_i}(x)} = x-t-a_rq_rt h_r(x), \qquad 1 \leq j \leq r-1, 
\end{equation}
and
\begin{equation}  \label{ratio2}
  \lim_{n_i \to \infty, n_i/N_i \to t} \frac{P_{\vec{n}_i+\vec{e}_r,N_i}(x)}{P_{\vec{n}_i,N_i}(x)} = x-a_r-t-a_rq_rt h_r(x). 
\end{equation}
A technical estimation (see Lemma \ref{lem} at the end of this section) implies that
\[   \lim_{n \to \infty, n/N \to t} \left| \frac{P_{\vec{n},N}(x)}{P_{\vec{n}+\vec{e}_j,N}(x)} -
     \frac{P_{\vec{n}-\vec{e}_j,N}(x)}{P_{\vec{n},N}(x)} \right| = 0, \qquad 1 \leq j \leq r, \]
uniformly on $K$, hence \eqref{ratio2} gives
\[   \frac{1}{h_r(x)} = x-a_r-t - a_r q_rt h_r(x). \]
If we put $g_r(x) = 1/h_r(x)$, then this gives a quadratic equation for $g_r(x)$, with solutions
\[ \frac{x-a_r-t \pm \sqrt{(x-a_r-t)^2-4a_rq_r t}}{2}.  \]
Since $h_r(x) = 1/x + \mathcal{O}(1/x^2)$, we need to choose the solution with the positive sign for $g_r(x)$.
This limit is independent of the subsequence that we selected, hence every convergent subsequence has the same limit, which
implies that the full sequence converges to this limit. This gives \eqref{gr}, and by using \eqref{ratio1} we easily find \eqref{gnotr}.

\end{proof}

The limit function $g_r(x)$ is the solution of a quadratic equation. In general, if $k \leq r$ of the parameters grow linearly with $N$, then
the limit function is expected to be the solution, which grows as $x$ when $x \to \infty$, of an algebraic equation of degree $k+1$.

For the asymptotic behavior of the zeros we have

\begin{theorem} \label{thm4}
Suppose $n_j = \lfloor q_jn\rfloor$, with $0 < q_j < 1$ and $\sum_{j=1}^r q_j = 1$, so that $|\vec{n}|/n \to 1$ if $n \to \infty$.
Consider Poisson distributions with parameters $(a_1,a_2,\ldots,a_{r-1},Na_r)$, i.e., the last parameter grows linearly with $N$.
Then for $t>0$ one has
\[   \lim_{n,N \to \infty,\ n/N \to t} \frac{1}{|\vec{n}|} \sum_{j=1}^{|\vec{n}|} f(x_{\vec{n},j}/N)
     = \frac{1}{t} \int_0^{(1-q_r)t} f(x)\, dx + q_r \int_{\alpha_t}^{\beta_t} v(x)f(x)\, dx  \]
for every bounded continuous function $f$ on $[0,\infty)$, where $(1-q_r)t \leq \alpha_t < \beta_t$ and $v$ is
a probability density on $[\alpha_t,\beta_t]$. 
\end{theorem}

\begin{proof}
We can start from equation \eqref{r}:
\[      \frac{P_{\vec{n},N}'(x)}{P_{\vec{n},N}(x)} =
      \frac{P_{\vec{n}-n_r\vec{e}_r,N}'(x)}{P_{\vec{n}-n_r\vec{e}_r,N}(x)}
       + \sum_{k=0}^{n_r-1}   \left( \frac{P_{\vec{n}-k\vec{e}_r,N}(x)}{P_{\vec{n}-(k+1)\vec{e}_r,N}(x)} \right)' 
          \Big/  \left(\frac{P_{\vec{n}-k\vec{e}_r,N}(x)}{P_{\vec{n}-(k+1)\vec{e}_r,N}(x)} \right). \]
The multiple orthogonal polynomial $P_{\vec{n}-n_r\vec{e}_r}$ is in fact the multiple Charlier polynomial with the $r-1$
parameters $(a_1,\ldots,a_{r-1})$, which do not depend on $N$. Hence we can use Theorem \ref{Thm2} which gives \eqref{r-1}.
We write the sum as an integral, as we did in the proof of Theorem \ref{Thm2}, but now we use Theorem \ref{thm3} to find
\begin{equation*}  
 \lim_{n,N \to \infty, n/N \to t}  
  \frac{1}{n_r} \sum_{k=0}^{n_r-1}   \left( \frac{P_{\vec{n}-k\vec{e}_r,N}(x)}{P_{\vec{n}-(k+1)\vec{e}_r,N}(x)} \right)' 
          \Big/  \left(\frac{P_{\vec{n}-k\vec{e}_r,N}(x)}{P_{\vec{n}-(k+1)\vec{e}_r,N}(x)} \right)  
  = \int_0^1 \frac{g_r(x,s)'}{g_r(x,s)}\, ds 
\end{equation*}
where 
\[  g_r(x,s) = \frac{x-a_r-(1-q_rs)t + \sqrt{(x-a_r-(1-q_rs)t)^2-4a_rq_r(1-s)t}}{2}  \]
and the prime is the derivative $d/dx$. The $g_r(x,s)$ is obtained from Theorem \ref{thm3} after the substitutions
\[   q_j \to \frac{q_j}{1-q_rs}, \qquad 1 \leq j \leq r-1, \quad q_r \to \frac{(1-s)q_r}{1-q_rs}, \]
so that $\sum_{j=1}^r q_j = 1$, 
\[   n_r \to  n_r - \lfloor n_r s \rfloor, \quad n \to n(1-q_rs), \quad t \to (1-q_rs)n.  \]
Observe that
\[    \frac{g_r(x,s)'}{g_r(x,s)} = \frac{1}{\sqrt{(x-a_r-(1-q_rs)t)^2-4a_rq_r(1-s)t}}  \]
and if we use the well known Stieltjes transform
\[   \frac{1}{\sqrt{x^2-1}} = \frac{1}{\pi} \int_{-1}^1 \frac{1}{x-y} \frac{dy}{\sqrt{1-y^2}}, \qquad x \in \mathbb{C} \setminus [-1,1], \]
then one finds
\[    \frac{g_r(x,s)'}{g_r(x,s)} = \frac{1}{\pi} \int_{\alpha(s)}^{\beta(s)}
    \frac{1}{x-y} \frac{dy}{\sqrt{4a_rq_r(1-s)t - (y-a_r-(1-q_rs)t)^2}},  \]
where 
\[   \alpha(s) = a_r+(1-q_rs)t - 2\sqrt{a_rq_r(1-s)t}, \quad \beta(s) = a_r+(1-q_rs)t + 2\sqrt{a_rq_r(1-s)t}.  \]
In order to write
\[ \int_0^1 \frac{g_r(x,s)'}{g_r(x,s)}\, ds  \]
as a Stieltjes transform, we need to change the order of integration in
\begin{equation}  \label{doubleint}
      \int_0^1 \frac{1}{\pi} \int_{\alpha(s)}^{\beta(s)}
    \frac{1}{x-y} \frac{dy}{\sqrt{4a_rq_r(1-s)t - (y-a_r-(1-q_rs)t)^2}} \, ds. 
\end{equation}
Observe that
\[  \alpha(0)=a_r+t-2\sqrt{a_rq_rt}, \quad \beta(0) = a_r+t+2\sqrt{a_rq_rt}, \quad \alpha(1)=\beta(1)=a_r+(1-q_r)t,  \]
and that the function $\beta$ is monotonically decreasing for $s \in [0,1]$.
We need to distinguish between two cases.
\begin{description}
  \item[Case 1:] $a_r\geq q_rt$. In this case the function $\alpha$ is monotonically increasing for $s \in [0,1]$, see Figure \ref{case1}.
\begin{figure}[ht]
\centering
\rotatebox{270}{\resizebox{3in}{!}{\includegraphics{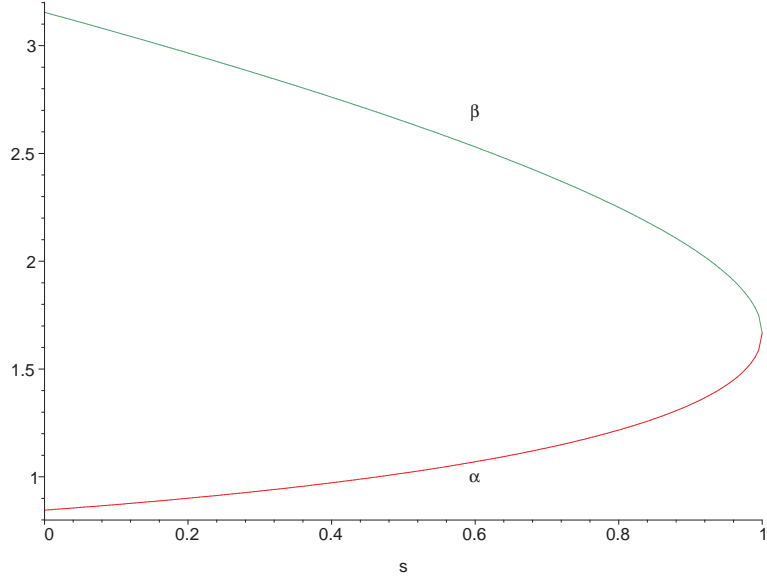}}}  
\caption{The functions $\alpha$ and $\beta$ for case 1}
\label{case1}
\end{figure}

If we define $\alpha_t=\alpha(0)$ and $\beta_t = \beta(0)$  then 
\[ \begin{cases}
   \alpha^{-1}(y) = \frac{-y-a_r+t + 2\sqrt{a_r(y-(1-q_r)t)}}{q_rt}, & \alpha_t \leq y \leq a_r+(1-q_r)t \\
   \beta^{-1}(y) =    \frac{-y-a_r+t + 2\sqrt{a_r(y-(1-q_r)t)}}{q_rt}, & a_r + (1-q_r)t \leq y \leq  \beta_t
   \end{cases}  \]
so that interchanging the order of integration in \eqref{doubleint} gives
\[   \frac{1}{\pi} \int_{\alpha_t}^{\beta_t} \frac{dy}{x-y} \int_0^{\frac{-y-a_r+t+2\sqrt{a_r(y-(1-q_r)t)}}{q_rt}}
      \frac{ds}{\sqrt{4a_rq_r(1-s)t - (y-a_r-(1-q_rs)t)^2}}.  \]
When we change the variable $s$ to a new variable $u$ by 
\[  s=\frac{-y-a_r+t+2u\sqrt{a_r(y-(1-q_r)t)}}{q_rt}, \]
then the integral simplies to
\[   \frac{1}{\pi q_rt} \int_{\alpha_t}^{\beta_t} \frac{dy}{x-y} \int^1_{\frac{y+a_r-t}{2\sqrt{a_r(y-(1-q_r)t)}}} \frac{du}{\sqrt{1-u^2}}. \]
This gives the weight function
\[   v(y) = \frac{1}{\pi q_rt} \int^1_{\frac{y+a_r-t}{2\sqrt{a_r(y-(1-q_r)t)}}} \frac{du}{\sqrt{1-u^2}}, \qquad
\alpha_t \leq y \leq \beta_t. \]
 An easy exercise gives that $(1-q_r)t \leq \alpha_t < \beta_t$. 
  \item[Case 2:] $a_r < q_rt$. In this case $\alpha$ has a global minimum on $]0,1[$ at $s=1-a_r/q_rt$, and the minimum
  is $(1-q_r)t$,  see   Figure \ref{case2}.
\begin{figure}[ht]
\centering
\rotatebox{270}{\resizebox{3in}{!}{\includegraphics{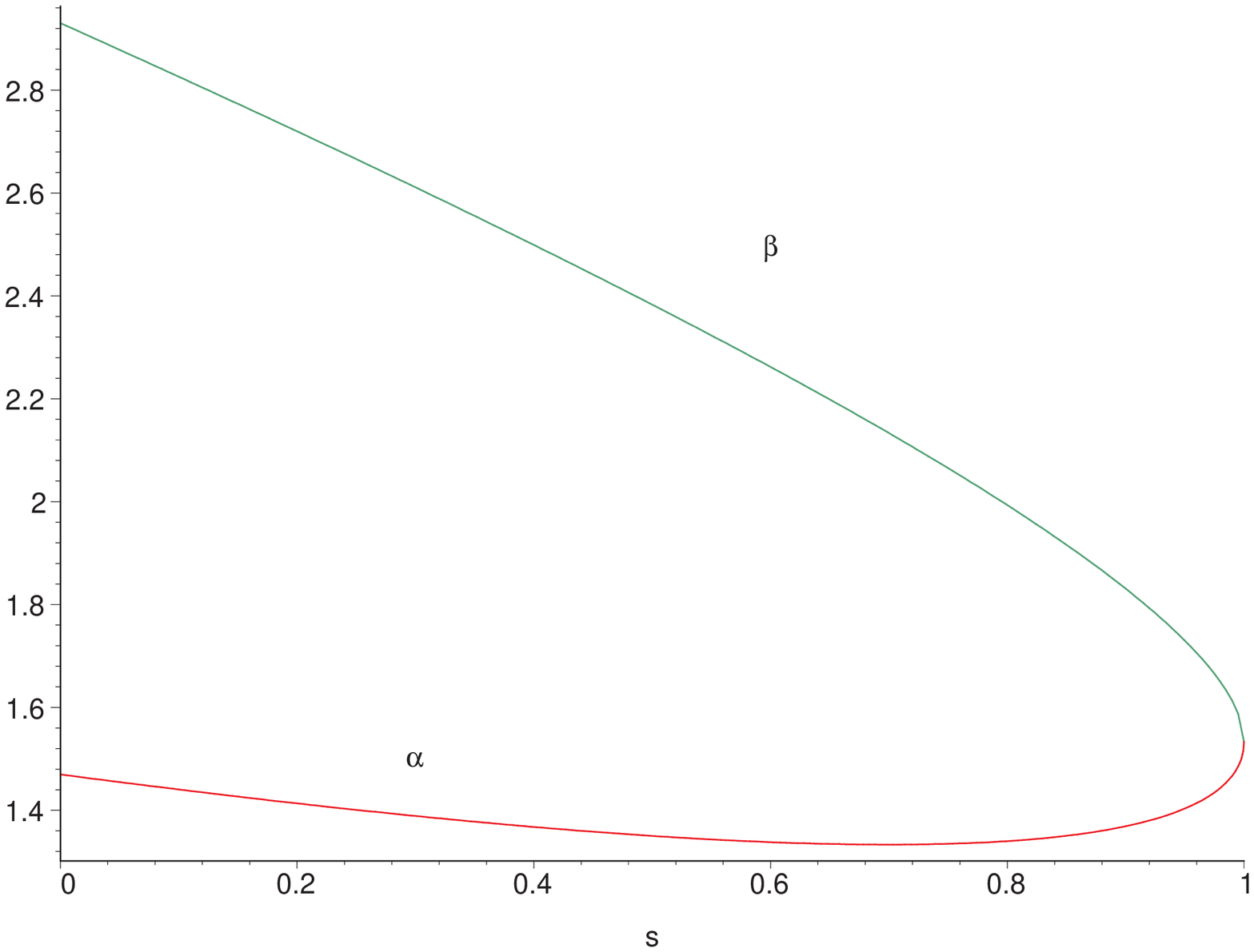}}}  
\caption{The functions $\alpha$ and $\beta$ for case 2}
\label{case2}
\end{figure}

Interchanging the order of the integrals in \eqref{doubleint} now gives two pieces
\begin{multline*}   
 \frac{1}{\pi} \int_{(1-q_r)t}^{\alpha(0)} \frac{dy}{x-y} 
    \int_{\frac{-y-a_r+t-2\sqrt{a_r(y-(1-q_r)t)}}{q_rt}}^{\frac{-y-a_r+t+2\sqrt{a_r(y-(1-q_r)t)}}{q_rt}}
      \frac{ds}{\sqrt{4a_rq_r(1-s)t - (y-a_r-(1-q_rs)t)^2}} \\
  + \frac{1}{\pi} \int_{\alpha(0)}^{\beta(0)} \frac{dy}{x-y} \int_0^{\frac{-y-a_r+t+2\sqrt{a_r(y-(1-q_r)t)}}{q_rt}}
      \frac{ds}{\sqrt{4a_rq_r(1-s)t - (y-a_r-(1-q_rs)t)^2}}.
\end{multline*}
The change of variable $s \to u$ with
\[  s=\frac{-y-a_r+t+2u\sqrt{a_r(y-(1-q_r)t)}}{q_rt} \]
now gives
\begin{equation*}   
 \frac{1}{\pi q_rt} \int_{(1-q_r)t}^{\alpha(0)} \frac{dy}{x-y} 
    \int_{-1}^1  \frac{du}{\sqrt{1-u^2}} 
  + \frac{1}{\pi q_rt} \int_{\alpha(0)}^{\beta(0)} \frac{dy}{x-y} 
   \int^1_{\frac{y+a_r-t}{2\sqrt{a_r(y-(1-q_r)t)}}}
      \frac{du}{\sqrt{1-u^2}}.
\end{equation*}
So if we now define $\alpha_t = (1-q_r)t$ and $\beta_t = \beta(0)$, then obviously $(1-q_r)t=\alpha_t < \beta_t$ and
the weight function becomes
\[  v(y) =  \frac{1}{q_rt}, \qquad (1-q_r)t \leq y \leq \alpha(0) \]
and
\[ v(y) = \frac{1}{\pi q_rt} \int^1_{\frac{y+a_r-t}{2\sqrt{a_r(y-(1-q_r)t)}}}
      \frac{du}{\sqrt{1-u^2}}, \qquad \alpha(0) \leq y \leq \beta(0).  \]
\end{description}
So in both cases we get
\begin{equation*}  
 \lim_{n,N \to \infty, n/N \to t}  
  \frac{1}{n_r} \sum_{k=0}^{n_r-1}   \left( \frac{P_{\vec{n}-k\vec{e}_r,N}(x)}{P_{\vec{n}-(k+1)\vec{e}_r,N}(x)} \right)' 
          \Big/  \left(\frac{P_{\vec{n}-k\vec{e}_r,N}(x)}{P_{\vec{n}-(k+1)\vec{e}_r,N}(x)} \right)  
  = \int_{\alpha_t}^{\beta_t} \frac{v(y)}{x-y} \, dy, 
\end{equation*}
and combining this with \eqref{r-1} gives
\begin{equation*}  
    \lim_{n,N \to \infty, n/N \to t} \frac{1}{|\vec{n}|} \frac{P_{\vec{n},N}'(x)}{P_{\vec{n},N}(x)}
    = \frac{1}{t} \int_0^{(1-q_r)t} \frac{1}{x-y} \, dy + q_r \int_{\alpha_t}^{\beta_t} \frac{v(y)}{x-y}\, dy,  
\end{equation*}
which gives the desired result in view of the Grommer-Hamburger theorem \cite{GH}.
\end{proof}

The first portion  of $(1-q_r)n$ of the zeros of $C_{\vec{n}}(Nx)$ are uniformly distributed on $[0,(1-q_r)t]$ and hence
the constraint that `between two positive integers there can be at most one zero' is in action and the zeros are forced
to approach the first $(1-q_r)n$ integers in $\mathbb{N}$. If $a_r \geq q_rt$ (case 1) then the last portion of $q_rn$ 
of the zeros have a different distribution on an interval $[\alpha_t,\beta_t]=[a_r+t-2\sqrt{a_rq_rt},a_r+t+2\sqrt{a_rq_rt}]$ to the right of
the interval $[0,(1-q_r)t]$ where the other zeros accumulate. This means that those last $q_rn$ zeros are less dense distributed and some of the intervals between two integers may be free of zeros.
If $a_r < q_rt$ (case 2) then some of the $q_rn$ last zeros are still uniformly distributed on $[(1-q_r)t,a_r+t-2\sqrt{a_rq_rt}]$ but the remaining zeros
are less dense distributed on $[a_r+t-2\sqrt{a_rq_rt},a_r+t+2\sqrt{a_rq_rt}]$ and this interval now touches the interval where the zeros
are uniformly distributed.
In fact, a transition occurs when $a_r=q_rt$ in the sense that the (scaled) zeros have a zero distribution on two disjoint intervals
when $a_r > q_rt$ and the zero distribution is supported on one interval when $a_r < q_rt$. Moreover, since
\[  \int_z^1 \frac{du}{\sqrt{1-u^2}} \sim C\sqrt{1-z}, \qquad z \to 1-  \]
and for $a_r > q_rt$
\[   1- \frac{y+a_r-t}{2\sqrt{a_r(y-(1-q_r)t)}} \sim \begin{cases}
                         C_1 (y-\alpha(0)), & y\to \alpha(0)+ \\
                         C_2 (\beta(0)-y), & y \to \beta(0)-
                   \end{cases}  \] 
we see that the density $v$ near the endpoints $\alpha(0)$ and $\beta(0)$ tends to zero as $\sqrt{y-\alpha(0)}$ and $\sqrt{\beta(0)-y}$, respectively
(see Figure \ref{density}, picture on the left, for $a_r=1$, $q_r=1/10$ and $t=1$).
\begin{figure}[ht]
\centering
\rotatebox{270}{\resizebox{2in}{!}{\includegraphics{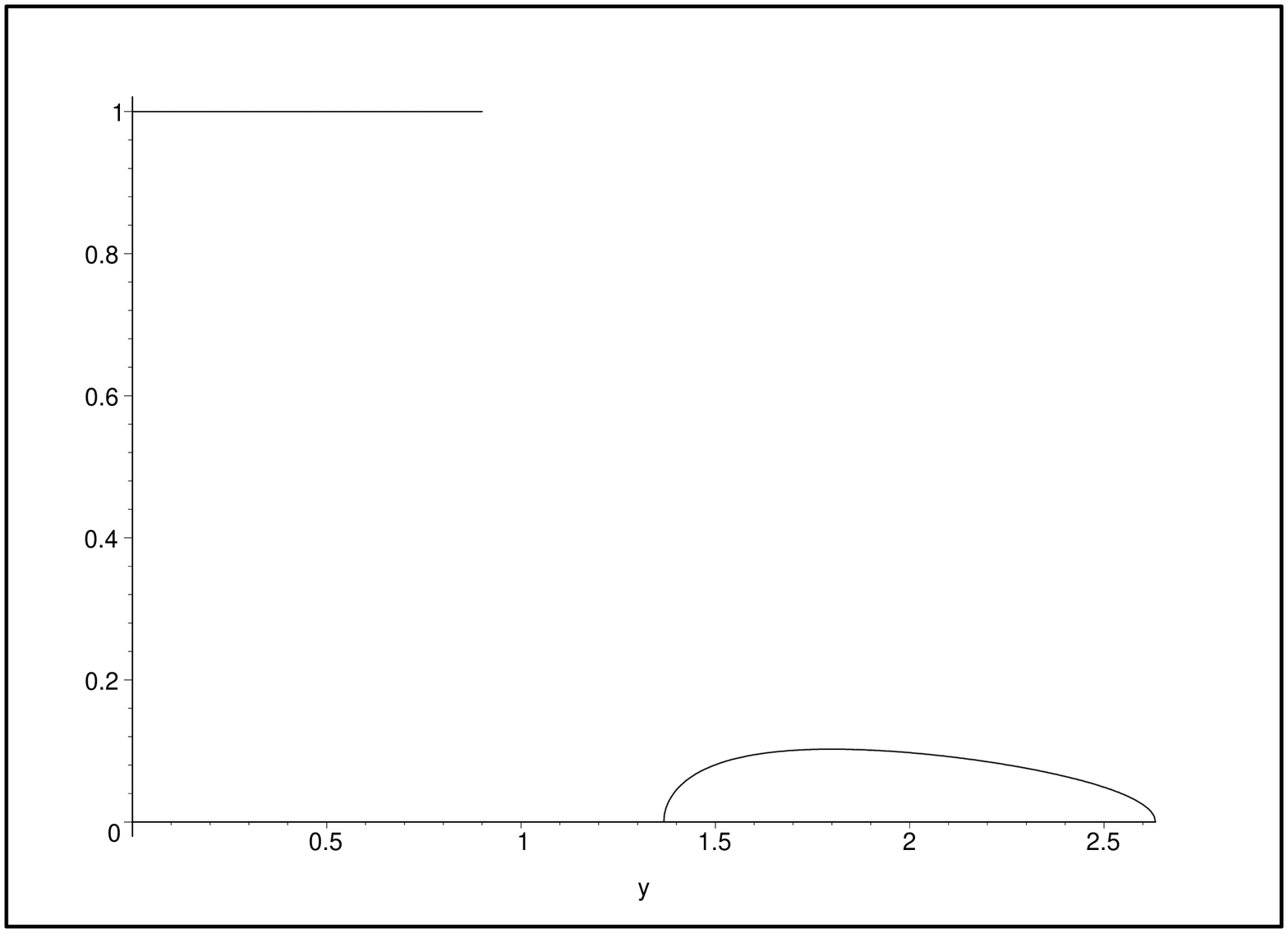}}}
\rotatebox{270}{\resizebox{2in}{!}{\includegraphics{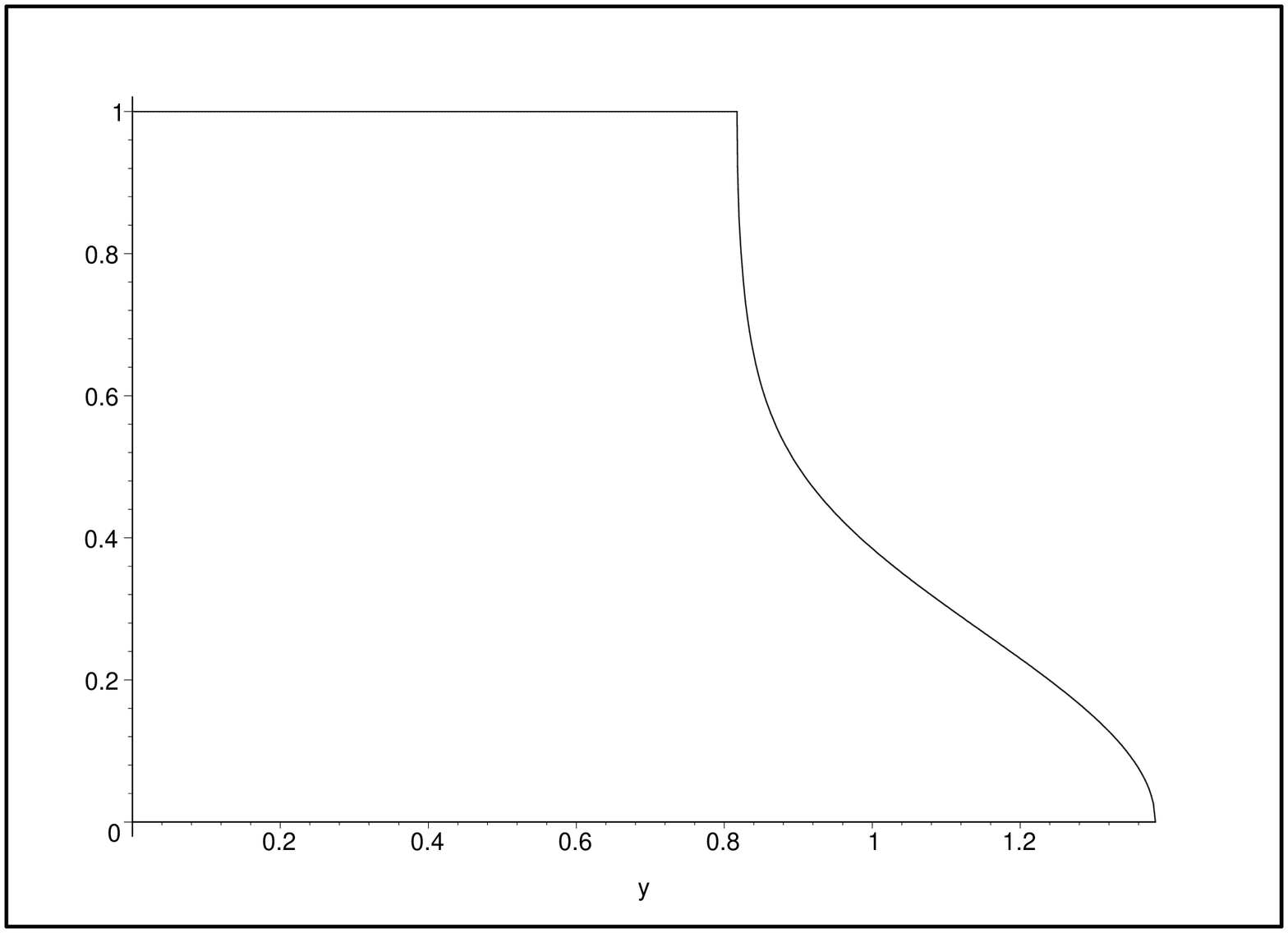}}}  
\caption{The density of the (scaled) zeros of multiple Charlier polynomials}
\label{density}
\end{figure}

For $a_r < q_rt$ we still have $v(y) \sim C \sqrt{\beta(0)-y}$ near the endpoint $\beta(0)$. The transition from uniform density to
non-uniform density occurs at $y = \alpha(0)$, but now 
\[  \frac{y+a_r-t}{2\sqrt{a_r(y-(1-q_r)t)}} \to -1 , \qquad y \to \alpha(0)+ \]
so that $v(y) \to 1/q_rt$ as $y \to \alpha(0)+$, and the density is continuous at the transition point $\alpha(0)$ (see Figure \ref{density},
picture on the right, for $a_r=1/10$, $q_r=1/5$ and $t=1$).

When $a_r=q_rt$ we have
\[  \frac{y+a_r-t}{2\sqrt{a_r(y-(1-q_r)t)}} \to 0 , \qquad y \to \alpha(0)+ \]
so that $v(y) \to 1/2q_rt$ as $y \to \alpha(0)+$, so that the density is not continuous at the transition point.

Such transitions also occur when $k < r$ of the parameters depend linearly on $N$. In that case the zeros of $C_{\vec{n}}(Nx)$
may accumulate on at most $k+1$ disjoint intervals. If all the parameters depend on $N$ (i.e., $k=r$) then the zeros accumulate
on at most $r$ disjoint intervals. The analysis for $k >1$ is more involved since this involves algebraic functions of order
$k+1$. 

One technical, but crucial, step in the proof of Theorem \ref{thm3} is the following.
\begin{lemma} \label{lem}
Let $P_{\vec{n},N}(x) = C_{\vec{n}}(Nx)/N^{|\vec{n}|}$, where $C_{\vec{n}}$ are the multiple Charlier polynomials with parameters
$(a_1,\ldots,a_{r-1},Na_r)$. Let $K$ be a compact set in $\mathbb{C}\setminus [0,\infty)$, then for every $k$ and $\ell$ with $1 \leq k,\ell \leq r$
one has, uniformly for $x \in K$
\[   \lim_{n \to \infty, n/N \to t} \left| \frac{P_{\vec{n},N}(x)}{P_{\vec{n}+\vec{e}_k,N}(x)} -
     \frac{P_{\vec{n}-\vec{e}_\ell,N}(x)}{P_{\vec{n}+\vec{e}_k-\vec{e}_\ell,N}(x)} \right| = 0. \]
\end{lemma}

\begin{proof}
From the recurrence relation, we have
\[   x = \frac{P_{\vec{n}+\vec{e}_k,N}(x)}{P_{\vec{n},N}(x)} + \frac{a_k+|\vec{n}|}{N} 
 + \sum_{j=1}^{r-1} \frac{n_ja_j}{N^2} \frac{P_{\vec{n}-\vec{e}_j,N}(x)}{P_{\vec{n},N}(x)} 
 + \frac{n_ra_r}{N} \frac{P_{\vec{n}-\vec{e}_r,N}(x)}{P_{\vec{n},N}(x)}  \]
when $1 \leq k \leq r-1$, and for $k=r$ we have
\[   x = \frac{P_{\vec{n}+\vec{e}_r,N}(x)}{P_{\vec{n},N}(x)} + \frac{Na_r+|\vec{n}|}{N} 
 + \sum_{j=1}^{r-1} \frac{n_ja_j}{N^2} \frac{P_{\vec{n}-\vec{e}_j,N}(x)}{P_{\vec{n},N}(x)} 
 + \frac{n_ra_r}{N} \frac{P_{\vec{n}-\vec{e}_r,N}(x)}{P_{\vec{n},N}(x)}.  \]
We will denote
\[   E_{\vec{n}}(x) = \sum_{j=1}^{r-1} \frac{n_ja_j}{N^2} \frac{P_{\vec{n}-\vec{e}_j,N}(x)}{P_{\vec{n},N}(x)}  \]
and the bound \eqref{ratiobound} then gives
\[   |E_{\vec{n}}(x)| \leq \frac{1}{\delta N^2} \sum_{j=1}^{r-1} n_j a_j \leq \frac{C|\vec{n}|}{\delta N^2}, \]
where $C>0$ is a constant (in fact on may take $\max_{1 \leq j \leq r-1} a_j$).
If we change $\vec{n}$ to $\vec{n}-\vec{e}_\ell$, then
\[   x = \frac{P_{\vec{n}+\vec{e}_k-\vec{e}_\ell,N}(x)}{P_{\vec{n}-\vec{e}_\ell,N}(x)} + \frac{a_k+|\vec{n}|-1}{N} 
 + E_{\vec{n}-\vec{e}_\ell}(x) 
 + \frac{(n_r-\delta_{r,\ell})a_r}{N} \frac{P_{\vec{n}-\vec{e}_r-\vec{e}_\ell,N}(x)}{P_{\vec{n}-\vec{e}_\ell,N}(x)}  \]
when $1 \leq k \leq r-1$, and for $k=r$ 
\[   x = \frac{P_{\vec{n}+\vec{e}_r-\vec{e}_\ell,N}(x)}{P_{\vec{n}-\vec{e}_\ell,N}(x)} + \frac{Na_r+|\vec{n}|-1}{N} 
 + E_{\vec{n}-\vec{e}_\ell}(x) + \frac{(n_r-\delta_{r,\ell})a_r}{N} \frac{P_{\vec{n}-\vec{e}_r-\vec{e}_\ell,N}(x)}{P_{\vec{n}-\vec{e}_\ell,N}(x)}.  \]
If we subtract the equations for $\vec{n}-\vec{e}_\ell$ from those with $\vec{n}$, then we find
\begin{multline*}   0 = \frac{P_{\vec{n}+\vec{e}_k,N}(x)}{P_{\vec{n},N}(x)} - \frac{P_{\vec{n}+\vec{e}_k-\vec{e}_\ell,N}(x)}{P_{\vec{n}-\vec{e}_\ell,N}(x)}
         + \frac{1}{N} + E_{\vec{n}}(x)- E_{\vec{n}-\vec{e}_\ell}(x) \\
       + \frac{n_ra_r}{N} \left( \frac{P_{\vec{n}-\vec{e}_r,N}(x)}{P_{\vec{n},N}(x)} -
     \frac{P_{\vec{n}-\vec{e}_r-\vec{e}_\ell,N}(x)}{P_{\vec{n}-\vec{e}_\ell,N}(x)} \right) + \frac{\delta_{r,\ell} a_r}{N}
        \frac{P_{\vec{n}-\vec{e}_r-\vec{e}_\ell,N}(x)}{P_{\vec{n}-\vec{e}_\ell,N}(x)}.  
\end{multline*}
We have
\[   |E_{\vec{n}}(x)- E_{\vec{n}-\vec{e}_\ell}(x)| \leq \frac{2C|\vec{n}|}{\delta N^2}  \]
and  we will take $|\vec{n}| \leq C_2 N$, therefore we have
\[    \left| \frac{P_{\vec{n}+\vec{e}_k,N}(x)}{P_{\vec{n},N}(x)} - \frac{P_{\vec{n}+\vec{e}_k-\vec{e}_\ell,N}(x)}{P_{\vec{n}-\vec{e}_\ell,N}(x)} \right| 
    \leq \frac{C_1}{N\delta} + C_2a_r \left|\frac{P_{\vec{n}-\vec{e}_r,N}(x)}{P_{\vec{n},N}(x)} -
     \frac{P_{\vec{n}-\vec{e}_r-\vec{e}_\ell,N}(x)}{P_{\vec{n}-\vec{e}_\ell,N}(x)} \right|, \]
where $C_1$ and $C_2$ are constants. If we use the bound \eqref{ratiobound}, then
\[     \left| \frac{P_{\vec{n}+\vec{e}_k,N}(x)}{P_{\vec{n},N}(x)} - \frac{P_{\vec{n}+\vec{e}_k-\vec{e}_\ell,N}(x)}{P_{\vec{n}-\vec{e}_\ell,N}(x)} \right| 
    \geq \delta^2 \left| \frac{P_{\vec{n},N}(x)}{P_{\vec{n}+\vec{e}_k,N}(x)} - \frac{P_{\vec{n}-\vec{e}_\ell,N}(x)}{P_{\vec{n}+\vec{e}_k-\vec{e}_\ell,N}(x)} \right|  \]
so that
\[   \left| \frac{P_{\vec{n},N}(x)}{P_{\vec{n}+\vec{e}_k,N}(x)} - \frac{P_{\vec{n}-\vec{e}_\ell,N}(x)}{P_{\vec{n}+\vec{e}_k-\vec{e}_\ell,N}(x)} \right|
   \leq \frac{C_1}{N\delta^3} + \frac{C_2a_r}{\delta^2} \left|\frac{P_{\vec{n}-\vec{e}_r,N}(x)}{P_{\vec{n},N}(x)} -
     \frac{P_{\vec{n}-\vec{e}_r-\vec{e}_\ell,N}(x)}{P_{\vec{n}-\vec{e}_\ell,N}(x)} \right|.  \]
If we  use the notation
\[     D_{\vec{n},k,\ell} = \left| \frac{P_{\vec{n},N}(x)}{P_{\vec{n}+\vec{e}_k,N}(x)} - \frac{P_{\vec{n}-\vec{e}_\ell,N}(x)}{P_{\vec{n}+\vec{e}_k-\vec{e}_\ell,N}(x)} \right|, \]
then this gives
\[   D_{\vec{n},k,\ell} \leq \frac{C_1}{N\delta^3} + \frac{C_2a_r}{\delta^2} D_{\vec{n}-\vec{e}_r,r,\ell}.  \]
Put
\[     D_{\vec{n},\ell} = \max_{1 \leq k \leq r} D_{\vec{n},k,\ell},   \]
then one has
\[     D_{\vec{n},\ell} \leq  \frac{C_1}{N\delta^3} + \frac{C_2a_r}{\delta^2} D_{\vec{n}-\vec{e}_r,\ell} . \]
Iterating this inequality gives
\[     D_{\vec{n},\ell} \leq \left( \frac{C_2a_r}{\delta^2} \right)^{n_r} D_{\vec{n}-n_r\vec{e}_r,\ell}
     + \frac{C_1}{N\delta^3} \sum_{j=0}^{n_r-1} \left( \frac{C_2a_r}{\delta^2} \right)^j .  \]
Now choose a compact $K'$ (with an accumulation point) far enough from $[0,\infty)$ so that $\delta$ is large and $C_2a_r/\delta^2 < 1$. 
Then for $x \in K'$
\[   D_{\vec{n},\ell} \leq \left( \frac{C_2a_r}{\delta^2} \right)^{n_r} D_{\vec{n}-n_r\vec{e}_r,\ell}
     + \frac{C_1}{N\delta^3} \frac{1}{1-C_2a_r/\delta^2}.  \]
The bound \eqref{ratiobound} gives $D_{\vec{n},\ell} \leq 2/\delta$, hence if we put
$\vec{n} = (\lfloor nq_1 \rfloor, \ldots, \lfloor nq_r \rfloor)$ and let $n,N \to \infty$ such that $n/N \to t >0$, then
\[     \lim_{n \to \infty, n/N \to t} D_{\vec{n},\ell} = 0 \]
uniformly for $x \in K'$. So we have convergence of $D_{\vec{n},\ell} \to 0$ uniformly on a set $K'$ with an accumulation point, but
then Vitali's theorem implies that $D_{\vec{n},\ell}$ converges to zero uniformly on every compact $K$ where a bound
\eqref{ratiobound} holds, hence for $K \subset \mathbb{C} \setminus [0,\infty)$. 
 \end{proof}

\section{Concluding remarks}
In this paper we have investigated the ratio asymptotic behavior of the multiple Charlier polynomials and from it we obtained the asymptotic 
distribution of the zeros, after proper rescaling. The next step is to find the asymptotic behavior of the polynomials $C_{\vec{n}}$
themselves: the strong asymptotic behavior or the uniform asymptotic behavior. As in the case of the usual Charlier polynomials, one will
need to look at different regions in the complex plane: away from the positive real line, on the oscillatory region where all the zeros are,
near the largest zero, near the origin, etc. One way to do this is to use the integral relation which can be obtained from the
multivariate generating function and to apply a steepest descent analysis (but for a multiple integral), as was done by Goh \cite{Goh} and Rui and Wong \cite{RW} for Charlier polynomials. Another way is to use the
Riemann-Hilbert problem (for $(r+1) \times (r+1)$ matrices) and the steepest descent method for oscillatory Riemann-Hilbert problems, as was done by Ou and Wong \cite{Ou} for Charlier polynomials.
One of the steps in that asymptotic analysis is to transform the Riemann-Hilbert problem to a normalized (at infinity) Riemann-Hilbert
problem, and this requires $g$-functions which are logarithmic potentials of the asymptotic zero distribution. Hence the results in Section 4 and
Section 5 (in particular Theorem \ref{thm4}) will be needed.

\bigskip

\begin{verbatim}
Walter Van Assche
Francois Ndayiragije
Department of Mathematics
Katholieke Universiteit Leuven
Celestijnenlaan 200B box 2400
BE-3001 Leuven
BELGIUM
walter@wis.kuleuven.be
francois.ndayiragije@wis.kuleuven.be
\end{verbatim}

\begin{thebibliography}{9}
\bibitem{ACV} J. Arves\'u, J. Coussement, W. Van Assche,
\textit{Some discrete multiple orthogonal polynomials},
J. Comput.\ Appl.\ Math.\ \textbf{153} (2003), 19--45.
\bibitem{BFPS} A. Borodin, P.L. Ferrari, M. Pr\"ahofer, T. Sasamoto,
\textit{Fluctuation properties of the TASEP with periodic initial configuration},
J. Statist.\ Phys.\ \textbf{129} (2007), 1055--1080.
\bibitem{BFS} A. Borodin, P.L. Ferrari, T. Sasamoto,
\textit{Two speed TASEP},
J. Statist.\ Phys.\ \textbf{137} (2009), 936--977.
\bibitem{Chihara} T.S. Chihara,
\textit{An Introduction to Orthogonal Polynomials},
Mathematics and its Applications \textbf{13}, 
Gordon and Breach, New York, 1978. 
\bibitem{Dun} T.M. Dunster, 
\textit{Uniform asymptotic expansions for Charlier polynomials}, 
J. Approx.\ Theory \textbf{112} (2001), 93-–133.
\bibitem{GH} J.S. Geronimo, T.P. Hill,
\textit{Necessary and sufficient condition that the limit of Stieltjes transforms is a Stieltjes transform},
J. Approx.\ Theory \textbf{121}, 54--60. 
\bibitem{Goh} W.M.Y. Goh, 
\textit{Plancherel-Rotach asymptotics for the Charlier polynomials}, 
Constr.\ Approx.\ \textbf{14} (1998), 151-–168.
\bibitem{HVA} M. Haneczok, W. Van Assche,
\textit{Interlacing properties of zeros of multiple orthogonal polynomials}, manuscript
\bibitem{Ismail} M.E.H. Ismail,
\textit{ Classical and Quantum Orthogonal Polynomials in One Variable},
Encyclopedia of Mathematics and its Applications \textbf{98}, Cambridge University Press, 2005 (paperback edition, 2009).
\bibitem{Joh} K. Johansson,
\textit{Discrete orthogonal polynomial ensembles and the Plancherel measure},
Ann.\ of Math.\ \textbf{153} (2001), 259--296.
\bibitem{KM} S. Karlin, J.L. McGregor,
\textit{Many server queueing processes with Poisson input and exponential service times},
Pacific J. Math.\ \textbf{8} (1958), 87--118.
\bibitem{KuiVA}  A.B.J. Kuijlaars, W. Van Assche, 
\textit{Extremal polynomials on discrete sets}, 
Proc.\ London Math.\ Soc.\ (3) \textbf{79} (1999), 191–-221.
\bibitem{Lee} D.W. Lee,
\textit{Difference equations for discrete classical multiple orthogonal polynomials},
J. Approx.\ Theory \textbf{150} (2008), 132--152. 
\bibitem{MaVA} M. Maejima, W. Van Assche, 
\textit{Probabilistic proofs of asymptotic formulas for some classical polynomials}, 
Math.\ Proc.\ Cambridge Philos.\ Soc.\ \textbf{97} (1985), 499–-510.
\bibitem{Vinet} H. Miki, L. Vinet, A. Zhedanov, 
\textit{Non-Hermitian oscillator Hamiltonians and multiple Charlier polynomials},
arXiv:1106.5243 [math-ph]
\bibitem{NikiSor} E.M. Nikishin, V.N. Sorokin,
\textit{Rational Approximations and Orthogonality}, 
Translations of Mathematical Monographs \textbf{92}, Amer.\ Math.\ Soc., Providence, RI, 1991.
\bibitem{Ou} Chun-Hua Ou, R. Wong,
\textit{Global Asymptotics of the Charlier polynomials via the Riemann-Hilbert method},
talk at ``Special Functions in the 21st Century: Theory and Applications'', Washington DC, April 6--8, 2011. \\
\texttt{http://math.nist.gov/\~{}DLozier/SF21/SF21slides/Ou.pdf}
\bibitem{PVA} K. Postelmans, W. Van Assche,
\textit{Multiple little $q$-Jacobi polynomials},
J. Comput.\ Appl.\ Math.\ \textbf{178} (2005), 361--375.
\bibitem{PR} M. Pr\'evost, T. Rivoal,
\textit{Remainder Pad\'e approximants for the exponential function},
Contr.\ Approx.\ \textbf{25} (2007), 109--123.
\bibitem{RW} Bo Rui, R. Wong, 
\textit{Uniform asymptotic expansion of Charlier polynomials}, 
Methods Appl.\ Anal.\ \textbf{1} (1994), 294-–313. 
\bibitem{WVA} W. Van Assche,
\textit{Difference equations for multiple Charlier and Meixner polynomials},
Proceedings of the Sixth International Conference on Difference Equations, CRC Press, Boca Raton, FL, 2004, pp.~549–-557.
\bibitem{WVA2} W. Van Assche,
\textit{Nearest neighbor recurrence relations for multiple orthogonal polynomials}, J. Approx.\ Theory \textbf{163} (2011), 1427--1448.
\bibitem{vD} E.A. van Doorn, A.I. Zeifman,
\textit{On the speed of convergence to stationarity of the Erlang loss system},
Queueing Syst.\ \textbf{63} (2009), 241--252.
\end{thebibliography}
\end{document}